\documentclass[11pt]{amsart}

\usepackage{amsmath}
\usepackage{amssymb}
\usepackage{amsthm}
\usepackage[all,cmtip]{xy}
\usepackage{color}

\theoremstyle{plain}
\newtheorem{thm}{Theorem}[section]
\newtheorem{prop}[thm]{Proposition}

\newtheorem{cor}[thm]{Corollary}

\newtheorem{prob}[thm]{Problem}

\theoremstyle{definition}
\newtheorem{defn}[thm]{Definition}
\newtheorem{eg}[thm]{Example}

\theoremstyle{remark}
\newtheorem{rem}[thm]{Remark}
\newtheorem{claim}[thm]{Claim}

\newcommand{\bC}{\ensuremath{\mathbb{C}}}

\newcommand{\bN}{\ensuremath{\mathbb{N}}}

\newcommand{\bP}{\ensuremath{\mathbb{P}}}
\newcommand{\bQ}{\ensuremath{\mathbb{Q}}}
\newcommand{\bR}{\ensuremath{\mathbb{R}}}

\newcommand{\bZ}{\ensuremath{\mathbb{Z}}}

\newcommand{\cA}{\ensuremath{\mathcal{A}}}

\newcommand{\cC}{\ensuremath{\mathcal{C}}}

\newcommand{\cH}{\ensuremath{\mathcal{H}}}

\newcommand{\cL}{\ensuremath{\mathcal{L}}}

\newcommand{\cO}{\ensuremath{\mathcal{O}}}

\newcommand{\cS}{\ensuremath{\mathcal{S}}}

\DeclareMathOperator{\Bs}{Bs}

\DeclareMathOperator{\Pic}{Pic}
\DeclareMathOperator{\Sing}{Sing}
\DeclareMathOperator{\Spec}{Spec}

\DeclareMathOperator{\Aut}{Aut}

\DeclareMathOperator{\NE}{NE}
 
\DeclareMathOperator{\Sym}{Sym} 
\DeclareMathOperator{\Supp}{Supp} 

\newcommand{\NEbar}{{\overline{\NE}}}

\begin{document}

\title
[Boundedness of CY hypersurfaces]
{On birational boundedness of \\ 
some Calabi--Yau hypersurfaces}

\author{Taro Sano}
\address{Department of Mathematics, Faculty of Science, Kobe University, Kobe, 657-8501, Japan}
\email{tarosano@math.kobe-u.ac.jp}

\maketitle

\begin{abstract} 
	We show the birational boundedness of anti-canonical irreducible hypersurfaces which form $3$-fold
	plt pairs. We also treat a collection of Du Val K3 surfaces which is birationally bounded but unbounded.  
	\end{abstract} 

\tableofcontents

\section{Introduction} 

In the classification of algebraic varieties, Calabi--Yau manifolds (CY manifolds for short) form an important class. 
It is not known whether $n$-dimensional CY manifolds form a bounded family for a fixed $n \ge 3$. 

On the other hand, in the 2-dimensional case, there are infinitely many projective families of K3 surfaces 
although they are analytically deformation equivalent. Reid observed that there are only 95 families of weighted K3 hypersurfaces (\cite[pp.300]{MR605348}, \cite[13.3]{MR1798982}). 
Inspired by this, we ask whether K3 surfaces in a 3-fold are bounded or not. 	
We show the following statement in this note.  

\begin{thm}\label{thm:K+D=0case}
Let $(X,D)$ be a plt pair such that $\dim X =3$, $D$ is irreducible and reduced, and $K_X+D \sim 0$. 
Then $D$ forms a birational bounded family. 
\end{thm}

An interesting feature is that $X$ can be unbounded as in Example \ref{eg:DPfib}. 
In fact, we study the birational boundedness of a prime divisor $D$ for a $3$-fold plt pair $(X,D)$ such that $K_X +D \equiv 0$ in Theorem \ref{thm:Kequiv0case}. 
It turns out that $D$ is birationally bounded unless $X$ is birational to a conic bundle over a Du Val surface $S$ with $K_S \sim 0$. 
The divisor $D$ can be unbounded as in the exceptional case as in Example \ref{eg:NikulinAbel}.  
The pair as above is called a plt CY pair in this note (Definition \ref{defn:pltCYpair}). 
CY pairs have been studied in several contexts of algebraic geometry (cf. \cite{MR3504536},  \cite{MR3539921}, \cite{MR3997127}, etc). 

The following example due to Oguiso forces us to use `birational boundedness' rather than `boundedness'  in Theorem \ref{thm:K+D=0case}.  

\begin{thm}(=Example \ref{eg:unbdd})
Fix any positive integer $d$. Then we have an unbounded collection of Du Val K3 surfaces which are birational contractions of smooth K3 surfaces of degree $2d$. 
\end{thm}

When $d=2$, the examples are birational contractions of some smooth quartic surfaces and 
infinitely many of them can be embedded into rational 3-folds (Remark \ref{rem:embedquarticcont}).  
Thus the statement in Theorem \ref{thm:K+D=0case} is optimal in a sense.  
	
Classically, examples of CY 3-folds are constructed by taking weighted or toric hypersurfaces. 
In Section 4, we ask whether CY hypersurfaces in rationally connected varieties form a bounded family. 	
We confirm that toric hypersurfaces form a bounded family in Corollary \ref{cor:torichypersurface}. 

Throughout this paper, we work over the complex number field $\bC$.
	
\section{Finiteness of anticanonical Calabi--Yau surfaces in a $3$-fold}

We follow the notation in \cite{MR1658959}. 

\begin{defn}\label{defn:pltCYpair}
We say that $(X,D)$ is a {\it plt Calabi--Yau (CY) pair} if $(X,D)$ is a plt pair such that $K_X +D \equiv 0$. 
A plt CY pair $(X,D)$ is called a {\it reduced plt CY pair} if $D$ is a reduced divisor. 
\end{defn}
Note that $X$ can be non-$\bQ$-Gorenstein, but the support of the round down $\lfloor D \rfloor$ of $D$ is normal (cf. \cite[Proposion 5.51]{MR1658959}). 
Note also that $X$ is $\bQ$-factorial in codimension 2 (cf. \cite[Proposition 9.1]{MR2854859}) and $K_X +D$ is torsion (cf. \cite[Corollary 10]{MR3022959}, \cite[Theorem 1.2]{MR3048544}). 

When $K_X +D \sim 0$ and $D$ is reduced, we have the following. 

\begin{prop}\label{prop:pltcanonical}
Let $(X,D)$ be a reduced plt CY pair such that $K_X +D \sim 0$. 

Then $D$ has only canonical singularities. 
If $X$ is $\bQ$-Gorenstein, then $X$ has only canonical singularities. 
\end{prop}

\begin{proof} 
We can take a log resolution $\mu \colon \tilde{X} \rightarrow X$ of $(X,D)$ such that 
\[
K_{\tilde{X}} + \tilde{D} = \mu^*(K_X +D) + \sum a_i E_i  
\]
for some integers $a_i \ge 0$, where $\tilde{D}$ is the strict transform of $D$ and $E_i$ is the exceptional divisor. 
Note that $a_i \ge 0$ since $K_X +D$ is Cartier. This implies that $X$ has only canonical singularities in codimension 2 (outside the non-$\bQ$-Gorenstein locus). 
In particular,  we see that $K_X$ is Cartier in codimension 2 and $(K_X+D)|_D = K_D$ is trivial. 
Thus, by restricting the equality to $\tilde{D}$, we see that $D$ has only canonical singularities. 
\end{proof}

The plt CY property is preserved by steps of the $K_X$-MMP as follows. 

\begin{prop}\label{prop:pltpreserve}
Let $(X,D)$ be a  reduced plt CY pair such that $X$ is projective and $\bQ$-factorial. 
Let 
$
\phi \colon X \dashrightarrow Y  
$
be a birational map which is a step of a $K_X$-MMP, that is, $\phi$ is either a divisorial contraction or a flip. Let $D_Y:= \phi_* D$. 
 
 Then the pair $(Y, D_Y)$ is also a plt  CY pair. 
\end{prop} 

\begin{rem}
We can not hope that $(Y, D_Y)$ is dlt when $(X, D)$ is so. 
Consider the pair $(\mathbb{P}^3, D)$ for a quartic surface $D$ with a simple elliptic singularity $p \in D$ and 
its blow-up $X_1 \rightarrow \mathbb{P}^3$ at $p$. 
Let $D_1 \subset X_1$ be the strict transform of $D$ and $E_1$ be the exceptional divisor. 
Then $(X_1, D_1+E_1)$ is a dlt CY pair, $X_1 \rightarrow \bP^3$ is a $K_X$-negative divisorial contraction and $(\bP^3, D)$ is lc and not dlt. 
\end{rem}

\begin{proof}
Since we have $D_Y = \phi_*(D) \in |{-}K_Y|_{\bQ}$, it is enough to show that $(Y,D_Y)$ is plt.  
Let $E$ be an exceptional divisor over $Y$ (hence over $X$). 
If $\phi \colon X \rightarrow Y$ is a divisorial contraction and $E$ is the $\phi$-exceptional prime divisor, 
we see that $E \not\subset \Supp D$ by the negativity lemma (cf. \cite[Lemma 3.6.2]{MR2601039}) since $D$ is $\phi$-ample. 
Hence we have $-1 < a(E, X, D) = a(E, Y, D_Y)$ since both $K_X +D$ and $K_Y +D_Y$ are trivial.  
Also when $\phi$ is a flip, we have the same equality by the same reason.  
Hence we see that both discrepancies are greater than $-1$, thus $(Y, D_Y)$ is also plt. 
 \end{proof}

The following is based on the argument in the e-mail from Chen Jiang. 

\begin{prop}\label{prop:differentbounded}
Let $n \in \bZ_{>0}$ and $I \subset [1,0] \cap \bQ$ be a DCC set.  
Let $(X, D)$ be an $n$-dimensional projective plt CY pair such that the coefficients of $D$ belong to $I$. Then we have the following. 
\begin{enumerate}
\item[(i)] $(X,D)$ is $\epsilon$-plt for some $\epsilon >0$ which only depends on $n$ and $I$, that is, for an exceptional divisor $E$ over $X$, the discrepancy 
$a(E;X,D) > -1+ \epsilon$.    
\item[(ii)] Assume that $\dim X=3$ and $D$ is reduced. 

Then $D$ is bounded except when $D$ has only Du Val singularities and 
$X$ is smooth in codimension $2$ around $D$. 

We have $(K_X + D) |_{D} = K_D$ in the exceptional case. 
\end{enumerate}
\end{prop}

\begin{proof}
(i) This can be shown by the same argument as \cite[Corollary 2.9]{Cerbo:aa} (In fact, (i) follows from \cite[Lemma 2.48]{MR3997127}). 
Suppose that there exists a plt CY pair $(X_n, D_n)$ which is $\epsilon_n$-plt for some $\epsilon_n >0$ such that $(\epsilon_n )_n$ is a decreasing sequence and 
 $\lim_{n \rightarrow \infty} \epsilon_n =0$. 
Then there is an extraction $\tilde{X}_n \rightarrow X_n$ of a divisor $E_n$ with $a(E_n; X_n, D_n) = -1+ \epsilon_n$ so that $(\tilde{X}_n, \tilde{D}_n + (1- \epsilon_n)E_n)$ 
satisfies the assumption of the global ACC \cite[Theorem 1.5]{MR3224718} since $I \cup \{1- \epsilon_n \mid n \in \bN \}$ is a DCC set. Thus 
$\{1- \epsilon_n \mid n \in \bN \}$ is a finite set and this is a contradiction. 

(ii) By the adjunction using the different, we have an equality
\[
K_X+D|_D = K_D + \sum_{i=1}^l b_i B_i
\]
as $\bQ$-divisors for some prime divisors $B_1, \ldots, B_l$. 
Note that $b_i$ belongs to some finite set $I_0$ by the global ACC \cite[Theorem 1.5]{MR3224718} 
since $b_i$ belongs to a DCC set $\{1- \frac{1}{n} \mid n \in \bN \}$. 
Suppose that $b_i \neq 0$ for some $i$. Then we see that $(D, \sum b_i B_i)$ is $\epsilon$-lc for some $\epsilon$ independent of $X$.  By \cite[Theorem 6.9]{MR1298994}, we see that $D$ belongs to a bounded family. 

Hence the problem is reduced to the case $K_X +D|_D = K_D$. This implies that $X$ is smooth at all codimension 1 points of $D$ by the local computation of the different 
(cf. \cite[Proposition 4.5 (1)]{MR3057950}). 
Thus we see that $K_D \equiv 0$. 
Such surfaces are bounded except when $D$ has only Du Val singularities by \cite[Theorem 6.9]{MR1298994}. 
\end{proof}

If a plt CY pair $(X,D)$ admits a del Pezzo fibration $X \rightarrow C$ over a curve, 
then $D$ belongs to a bounded family as follows.  
(Note that $C$ is either $\bP^1$ or an elliptic curve by the canonical bundle formula. )

\begin{prop}\label{prop:delpezzofibbdd}
Let $(X, D)$ be a projective $\bQ$-Gorenstein $3$-fold plt CY pair 
with a fiber space $\phi \colon X \rightarrow C$ over a smooth curve $C$ 
such that $D$ is irreducible, reduced and  $\phi$-ample.  

Then there exist a positive integer $N$ and an ample line bundle $\cH$ on $D$ 
such that $N$ is independent of $X$ and $\cH^2 \le N$, thus such $D$'s form a bounded family. 
\end{prop}

\begin{proof} 
Note first that $(X,D)$ is $\epsilon$-plt by Proposition \ref{prop:differentbounded} (i) for some $\epsilon >0$ 
and the general fiber $X_p$ over $p \in C$ of $\phi$ is an $\epsilon$-lc log del Pezzo surface. 
By Proposition \ref{prop:differentbounded}(ii), it is enough to consider the case where $D$ has only Du Val singularities and $X$ is 
smooth in codimension 2 around $D$. 
By this, the restriction $-K_X|_D$ is determined as a Weil divisor. 

\begin{claim}
There exists a positive integer $m$ such that $m$ is independent of $X$ and $m L$ is a Cartier divisor 
for all Weil divisor $L$ on $D$. 
\end{claim}

\begin{proof}[Proof of Claim]
The claim follows since there are finitely many possibilities for the singularities on $D$ (cf. \cite[(4.8.1)]{MR1941620}). 
Let $\nu_D \colon \tilde{D} \rightarrow D$ be the minimal resolution. If $D$ is singular, then $\tilde{D}$ is either a K3 surface or an Enriques surface.  
Then the number of the $\nu_D$-exceptional $(-2)$-curves is less than $\rho(\tilde{D}) \le 20$ (or $< 10$ if $\tilde{D}$ is Enriques) since the exceptional curves are linearly independent in $\Pic \tilde{D}$. 
\end{proof}


We shall find an ample divisor of the form $m(-K_X+aF)|_D$ for a fiber $F:= \phi^{-1}(p)$. 
The point is that $a$ can be unbounded as in Example \ref{eg:DPfib}, but the degree of the divisor is bounded. 

Let $\phi_D:= \phi|_D \colon D \rightarrow C$ and $F_D:= \phi_D^{-1}(p)$ be its fiber over $p \in C$. 
Then $\phi_D$ is an elliptic fibration since, for a general $p \in C$, we have $F_D \in |{-}K_F|_{\bQ}$ for a log del Pezzo surface $F$ and we check $h^0(F_D, \cO_{F_D}) \simeq \bC$. 

Let $\cL_{D}:= m(-K_X|_{D})$ be the restricted divisor which is $\phi_{D}$-ample. Let 
\[
\alpha:= \min \{a \in \mathbb{Z} \mid h^0(D, \cL_D + a F_D) \neq 0 \}. 
\]
Then we have an exact sequence 
\[
0= H^0(D, \cL_D + (\alpha-1)F_D) \rightarrow H^0(D, \cL_D + \alpha F_D) 
\rightarrow H^0(F_D, (\cL_D+ \alpha F_D)|_{F_D}).  
\]
Note that $(\cL_D+ \alpha F_D)|_{F_D} = -mK_X |_{F_D}$ and 
its degree is $m(-K_F^2) =: md$, 
where $F$ is a general fiber of $\phi$ which is an $\epsilon$-lc del Pezzo surface of degree $d$. 
Indeed, we have 
\[
-K_X \cdot F_D = -K_X \cdot D \cdot F = (-K_X)^2 \cdot F = (-K_F)^2 =d. 
\]
Note that $d \le \delta$ for some integer $\delta=\delta_{\epsilon}$ determined by $\epsilon$ (the maximal integer degree of $\epsilon$-lc del Pezzo surfaces. See \cite{MR3158579} for the optimal bound. ).  
Since $F_D$ is an elliptic curve, we have $h^0(F_D, (\cL_D+ \alpha F_D)|_{F_D}) = md$.  
Thus, by the above exact sequence, we see that 
\begin{equation}\label{eq:h0ineq}
h^0(D, \cL_D + \alpha F_D) \le md. 
\end{equation}


\begin{claim}\label{claim:amplealpha}
The Cartier divisor
$\cL_D + (k+\alpha) F_D$ is ample for $k > 2\delta m$. 
\end{claim}

\begin{proof}[Proof of Claim] 
Let $\nu_D \colon \tilde{D} \rightarrow D$ be the minimal resolution of $D$ and 
$F_{\tilde{D}}:= \nu_D^*(F_D)$ be the pull-back, and  
$\phi_{\tilde{D}}:= \nu_D \circ \phi_D \colon \tilde{D} \rightarrow C$ 
be the composition.  

Let $\nu_D^*(\cL_D + \alpha F_D) = M + E$ 
be the decomposition to the mobile part $M$ and fixed part $E$.   
We can write $E = \sum_{i=1}^l a_i C_i$ for some $a_i \ge 0$ and $(-2)$-curves $C_1, \ldots, C_l$ so that 
$C_1, \ldots, C_{l'}$ are $\phi_{\tilde{D}}$-horizontal and 
$C_{l'+1}, \ldots, C_l$ are $\phi_{\tilde{D}}$-vertical. 
Note that 
\[
md = \nu_D^*(\cL_D+ \alpha F_D) \cdot F_{\tilde{D}} \ge E \cdot F_{\tilde{D}} = (\sum_{i=1}^l a_i C_i) \cdot F_{\tilde{D}} \ge 
\sum_{i=1}^{l'} a_i 
\]
since $C_i$ is vertical for $i >l'$. 
Hence we obtain $$a_i \le \delta m \ \ (i=1,\ldots, l').$$ 

In order to check $\cL_D + (k+\alpha) F_D$ is nef, it is enough to check  
$$\nu_D^*(\cL_D+(\alpha+k) F_D) \cdot C_i \ge 0$$ for $i=1, \ldots, l'$ 
since $\cL_D$ is $\phi_D$-ample.  
For $k \ge 2\delta m$, we have 
\begin{multline*}
\nu_D^*(\cL_D+(\alpha+k) F_D) \cdot C_i = (M + E + k F_{\tilde{D}}) \cdot C_i \\
\ge (a_i C_i + k F_{\tilde{D}}) \cdot C_i = -2 a_i + k(F_{\tilde{D}} \cdot C_i) \ge -2\delta m + k (F_{\tilde{D}} \cdot C_i) \ge 0.  
\end{multline*}
since $C_i$ is horizontal and $F_{\tilde{D}} \cdot C_i \ge 1$. 
Thus $\cL_D+(\alpha+k) F_D$ is nef for $k \ge 2\delta m$, thus ample when $k > 2\delta m$. 
\end{proof}

For a positive integer $\beta$ and a divisor $\cL_{\beta}:= \cL_D + (\alpha+ \beta) F_D$, 
we have an exact sequence 
\[
0 \rightarrow H^0(D, \cL_{\beta}) \rightarrow H^0(D, \cL_{\beta+1}) 
\rightarrow H^0(F_D, \cL_{\beta+1}|_{F_D}). 
\]
By $h^0(F_D, \cL_{\beta+1}|_{F_D}) = h^0(F_D, \cL_D|_{F_D}) = md$ as before,  we have 
\[
h^0(D, \cL_{\beta +1}) \le h^0(D, \cL_{\beta}) +md.
\] 
By this and (\ref{eq:h0ineq}), we obtain 
\[
h^0(D, \cL_{2\delta m +1}) \le md + (2\delta m+1) md = 2\delta m^2d + 2md \le 2m^2 \delta^2 + 2m\delta.
\] 
Since $\cL_{2\delta m+1}$ is ample, we have $h^i(D, \cL_{2\delta m+1}) =0$ for $i=1,2$. 
Since $\cL_{2\delta m+1}$ is Cartier, we obtain 
\[
h^0(D, \cL_{2\delta m+1}) = \chi (D, \cL_{2\delta m +1}) = \chi_D + \frac{(\cL_{2\delta m+1})^2}{2}, 
\] 
where $\chi_D:= \chi(D, \cO_D)=0,1,2$ since $D$ is either a (Du Val) K3 surface, Enriques surface or abelian surface. 

Thus we see that $\cL_{2\delta m+1}^2$ is bounded by the constant 
$2(2m^2 \delta^2 + 2m\delta-\chi_D)$ 
and $\cH:= \cL_{2\delta m+1}$ has the required property. 
By \cite[Lemma 3.7 (1)]{MR1298994}, we see that $D$ forms a bounded family. 

\end{proof}

\begin{rem}
When $D$ is an abelian surface, we have the same statement as Claim \ref{claim:amplealpha} for $k >0$ since 
an effective divisor on $D$ is nef. 
\end{rem}

\begin{eg}\label{eg:conicbundle}
There are infinitely many examples of conic bundles with smooth anticanonical members in \cite[Example 20]{MR3743104}. 
Let $\bP:= \bP (\cO_{\bP^2} \oplus \cO_{\bP^2}(3) \oplus \cO_{\bP^2}(c))$ for $c \ge 3$ 
and $X:=X_c \in |\cO_{\bP}(2)|$ be a smooth member. Then $\phi \colon X \rightarrow \bP^2$ is a conic bundle and 
$|{-}K_X|$ contains a smooth member $D$. Since $D$ is also an anticanonical member of $\bP(\cO_{\bP^2} \oplus \cO_{\bP^2}(3))$, 
we see that $D$ is bounded with a polarization $\cO_{\bP}(1)|_{D}$ of degree $18$. 
We see that $\rho(X)=2$ by the Lefschetz type theorem \cite[Theorem 2]{MR2219849} and check that the collection $\{X_c\}_{c=1,2,\ldots}$ is unbounded. 
Indeed, a nef and big divisor on $X$ can be written as 
$$H_{a,b}:=-aK_X+ bF = a(-K_X + cF)+ (b-ca) F,$$ 
where $a, b \in \mathbb{Z}_{>0}$ satisfy $b \ge ca$ and $F:= \phi^* \cO_{\bP^2}(1)$.  
Thus we compute 
$$H_{a,b}^3 \ge (-K_X+cF)^3 = 2 (\cO_{\bP}(1)^4) =  2(c^2 + 3c +9)$$ 
by using $H \cdot (H-3f) \cdot (H-cf)=0$ for $H:= \cO_{\bP}(1)$ and $f:=\pi^* \cO_{\bP^2}(1)$ for $\pi \colon \bP \rightarrow \bP^2$. 
Indeed, since we have $H^3 = (c+3) H^2 \cdot f -3c H \cdot f^2$ and $H^2 \cdot f^2 =1$, we obtain 
\begin{multline*}
H^4 = H (H^3) = H ((c+3) H^2 \cdot f -3c H \cdot f^2) = (c+3) H^3 \cdot f -3c (H^2 f^2) \\ 
= (c+3) ((c+3)H^2 \cdot f - 3c H\cdot f^2)\cdot f - 3c \cdot 1 = (c+3)^2 - 3c = c^2 + 3c +9. 
\end{multline*} 
Hence we see the unboundedness of $X_c$. 

Moreover, we check that the collection $\{X_c \mid c \ge 3 \}$ is birationally unbounded by the same argument as \cite{MR1939018}. 
Indeed, the discriminant curve $B_c \subset \bP^2$ of $\phi_c \colon X_c \rightarrow \bP^2$ has degree $2c +6$ as \cite[Example 20]{MR3743104}, thus $4K_{\bP^2} + B_c$ is effective when $c \ge 3$. Hence the conic bundle $\phi_c \colon X_c \rightarrow \bP^2$ is birationally rigid (cf. \cite[Theorem 4.2]{MR1798984}). Then we can use the argument in \cite[Section 3]{MR1939018} to show that $\{X_c \mid c \ge 3 \}$ is birationally unbounded. 

\end{eg}

\begin{eg}\label{eg:DPfib}
There also exist infinitely many examples of del Pezzo fibrations $X \rightarrow \bP^1$ such that $X$ is smooth and $|{-}K_X|$ contains a smooth member. 
Let $$X:= X_n \subset \bP:= \bP_{\bP^1}(\cO\oplus \cO \oplus \cO(2) \oplus \cO(n))$$ be a smooth member of $|\cO_{\bP}(3)|$. 
Then the induced projection $\phi \colon X \rightarrow \bP^1$ is a del Pezzo fibration and 
$|{-}K_X| = |\cO_{\bP}(1) \otimes \phi^* \cO(-n)|$ contains a smooth member $S$. 
We see that $S$ is isomorphic to an anticanonical member of $\bP_{\bP^1}(\cO \oplus \cO \oplus \cO(2))$ 
and has a polarization of the degree independent of $n$. However, the collection $\{ X_n\}_{n \in \bN}$ is not  bounded. 
Indeed, we see that $\Pic X = \bZ (-K_X) \oplus \bZ (F)$ for $F:= \phi^* \cO_{\bP^1}(1)$ as above, and 
a nef and big line bundle 
$$G_{a,b}:= a(-K_X) + b F = a(-K_X + nF) + (b-na)F$$ 
should satisfy $b \ge na$. 
Thus we see the unboundedness of $X_n$ by  computing 
\[
G_{a,b}^3 \ge (-K_X + nF)^3 
= 3n +6
\]
since $0= H^2 \cdot (H-2f) (H-nf) =H^2(H^2-(n+2)H\cdot f +2n f^2) = H^4 - (n+2)H^3 \cdot f = H^4-(n+2)$, where 
$H:= \cO_{\bP}(1)$ and $f$ is the fiber class. 

For an elliptic curve $C$ and a positive integer $d$, consider $\bP_C:= \bP(\cO_C \oplus \cO_C \oplus \cO_C \oplus \cO_C(dP) )$ and a smooth member $X_d \in |\cO_{\bP_C}(3)|$. 
Then $X_d \rightarrow C$ is a del Pezzo fibration and $S_d \in | {-}K_{X_d}|$ is an abelian surface with a bounded polarization. 
We check the unboundedness of $X_d$ by a similar calculation as above.  
\end{eg}

The following implies Theorem \ref{thm:K+D=0case}. 

\begin{thm}\label{thm:Kequiv0case}
Let $(X,D)$ be a projective $3$-fold plt CY pair such that $D$ is irreducible and reduced. 
Then $D$ is birationally bounded unless all of the following hold: 
\begin{enumerate}
\item $K_X+D \not\sim 0$, but $2(K_X+D) \sim 0$. 
\item $X$ is birational to a conic bundle $Y \rightarrow S$ such that 
 $S$ is either a Du Val K3 surface or an abelian surface. 
 \item For the strict transform $D_Y \subset Y$ of $D$, the induced morphism $D_Y \rightarrow S$ is \'{e}tale in codimension 1
\end{enumerate}
In particular, Theorem \ref{thm:K+D=0case} holds. 
\end{thm}

\begin{proof}
By taking a small $\bQ$-factorial modification (cf. \cite[Corollary 1.37]{MR3057950}), we may assume that $X$ is $\bQ$-factorial.  

Let $\phi \colon X \dashrightarrow X_m$ be a birational map induced by a $K_X$-MMP and $\phi_D \colon D \dashrightarrow D_m$ be the birational map induced by $\phi$. 
We also have a Mori fiber space $\phi_m \colon X_m \rightarrow S$. Note that $(X_m, D_m)$ is also a  plt CY pair by Proposition \ref{prop:pltpreserve}. 
It is enough to consider the case where $D_m$ has only Du Val singularities by Proposition \ref{prop:differentbounded}(ii). 
The problem is to bound such $D_m$. 

Consider the case $\dim S =0$. Then $X_m$ is a $\epsilon$-lc Fano 3-fold for some $\epsilon >0$ by Proposition \ref{prop:differentbounded}, thus it is bounded by \cite[Theorem 1.1]{MR4224714} and $D_m$ is also bounded. 

Next consider the case $\dim S =1$. Then $X_m \rightarrow S$ is a del Pezzo fibration and $D_m$ is bounded by Proposition \ref{prop:delpezzofibbdd}. 

Next consider the case where $\dim S =2$ and  the induced morphism $ D_m \rightarrow S$ is of degree $2$ and branched along a curve. 
Then $(S,\frac{1}{2}R)$ is a $\frac{1}{2}$-lc CY pair (cf. \cite[Proposition 5.20]{MR1658959}), where $R \in | {-}2K_S|$ is 
the branch divisor of the double cover $\pi_m \colon D_m \rightarrow S$ (or its Stein factorization). 
Then $(S,\frac{1}{2}R)$ is log bounded by \cite[Theorem 6.9]{MR1298994}. 
Thus $D_m$ is also bounded since it is a crepant modification of the double cover of $S$ branched along $R$ 
(For a polarization $H$ on $S$ with the bounded degree, $\pi_m^*H$ gives a quasi-polarization on $D_m$ with the bounded degree).  

Finally consider the case where $\dim S =2$ and $\pi_m \colon D_m \rightarrow S$ is \'{e}tale in codimension 1. 
Then we see that $K_S \equiv 0$. Thus $S$ and $D_m$ are bounded unless $S$ has only Du Val singularities by \cite[Theorem 6.8]{MR1298994}. 
Since we are interested in the birational boundedness of $D$, it is enough to assume $K_S \sim 0$, that is, $S$ is either a Du Val K3 surface or an abelian surface
since Enriques surfaces and bielliptic surfaces are bounded. 
Hence the problem is reduced to the following claim.   

\begin{claim}
In the above setting, assume that $S$ is a Du Val K3 surface or an abelian surface. Then we have the following. 
\begin{enumerate}
\item[(i)] $K_{X_m} + D_m \not\sim 0$. 
\item[(ii)] $2(K_{X_m} +D_m) \sim 0$. 
\end{enumerate}
\end{claim}

\begin{proof}[Proof of Claim]
Let $X:= X_m$ and $D:= D_m$ with a conic bundle $\phi \colon X \rightarrow S$. Note that $\phi_D:= \phi|_D$ is \'{e}tale in codimension 1 and, if $S$ is an abelian surface, then $\phi_D$ is \'{e}tale by the purity of the branch locus. 

\noindent(i) Suppose that $K_X + D \sim 0$ and we shall find a contradiction.  Since we have the usual adjunction $K_X+D|_D = K_D$ and $\cO_X(K_X)$ is $S_2$, we obtain an exact sequence 
\[
0 \rightarrow \cO_X(K_X) \rightarrow \cO_X(K_X +D) \rightarrow \cO_D(K_D) \rightarrow 0. 
\] 
Since the restriction $H^0(X, K_X +D) \rightarrow H^0(D, K_D)$ is surjective, 
we obtain the exact sequence 
\[
0 \rightarrow H^1(X, K_X) \rightarrow H^1(X, \cO_X) \xrightarrow{\alpha} H^1(D, \cO_D). 
\]
By the Serre duality and the Leray spectral sequence, we obtain  
\[
H^1(X, K_X) \simeq H^2(X, \cO_X)^* \simeq H^2(S, \cO_S)^* \simeq \bC
\] 
and $H^1(X, \cO_X) \simeq H^1(S, \cO_S)$. Note that $R^i \phi_* \cO_X =0$ by the Kawamata-Viehweg vanishing since $-K_X$ is $\phi$-ample. If $S$ is a Du Val K3 surface, then we have $H^1(S, \cO_S) =0$ and this contradicts the above exact sequence. 
If $S$ is an abelian surface, then we check that $\alpha$ in the exact sequence is injective. Indeed, $\alpha$ can be regarded as $\phi_D^* \colon H^1(S, \cO_S) \rightarrow H^1(D, \cO_D)$ and this is an isomorphism since $\phi_D$ is \'{e}tale. This again contradicts the above exact sequence. 
Thus we see that $K_X +D$ is not trivial. 

\vspace{2mm}

\noindent(ii) 
Let $m \in \bZ_{>1}$ be a minimal integer such that $m(K_X+D) \sim 0$ and let $\Pi \colon X':= \Spec \bigoplus_{i=0}^{m-1} \cO_X(i(K_X+D)) \rightarrow X$ be the cyclic cover defined by an isomorphism $\cO_X(m(K_X + D)) \simeq \cO_X$. Then $D':= \Pi^{-1}(D)$ satisfies that $D' \simeq \Spec \bigoplus_{i=0}^{m-1} \cO_X(i(K_X+D)|_D) \simeq \Spec \bigoplus_{i=0}^{m-1} \cO_X(i K_D)$. 
By $K_D \sim 0$, we see that $D'$ is a disjoint union of $m$ copies of $D$. By $K_{X'} + D' \sim 0$ and \cite[Proposition 4.37 (3)]{MR3057950}, we see that $m=2$, that is, $2(K_X+D) \sim 0$. 
\end{proof}

This finishes the proof of Theorem \ref{thm:Kequiv0case}. 
\end{proof}	

The case where $D_m \rightarrow S$ is \'{e}tale really occurs as follows. 
We also have examples where $D_m$ can be any abelian surface, thus gives examples of birationally unbounded $D$ in Theorem \ref{thm:Kequiv0case} by Claim \ref{claim:unbddK3}.  

\begin{eg}\label{eg:Enriques}
Let $S$ be an Enriques surface and $X:= \bP_S(\cO_S \oplus \omega_S)$. 
Then the linear system $|{-}K_X| = |\cO_{\bP}(2)|$ is free. 
Indeed, it contains two members $2 \sigma_0, 2 \sigma_{\infty}$ with disjoint support, 
where  $\sigma_0, \sigma_{\infty}$ are the sections corresponding to two surjections $\cO \oplus \omega_S \twoheadrightarrow \cO, 
\cO \oplus \omega_S \twoheadrightarrow \omega_S$. 
Then we see that a general member $D \in |\cO_{\bP}(2)|$ is irreducible since we have an exact sequence 
\[
H^0( \cO_{\bP}) \rightarrow H^0(\cO_D) \rightarrow H^1(\cO_{\bP}(-2))
\]
and obtain $H^0(D, \cO_D) \simeq \bC$ by 
$$H^1(\cO_{\bP}(-2)) = H^1(\omega_{\bP}) \simeq H^2(\cO_{\bP})^{*} \simeq H^2( \cO_S) =0. $$ 
Then, since there is an \'{e}tale double cover $D \rightarrow S$, we see that $D$ is a K3 surface.  
It is well-known that Enriques surface has a polarization $H$ such that $H^2 =2$, thus Enriques surfaces form a bounded family. 

We can construct a similar example from any abelian surface $A$ and its translation $\tau \in \Aut A$ by a 2-torsion point on $A$. 
Note that the quotient morphism $q \colon A \rightarrow A/\tau$ is \'{e}tale and $\bar{A}:= A/\tau$ is also an abelian surface. 
Let $Y:= \bP_{\bar{A}}(\cO \oplus \cL)$, where $q_* \cO_A \simeq \cO_{\bar{A}} \oplus \cL$. Then $|\cO_{\bP}(2)|$ is free and contains a smooth member $\Delta \simeq A$ as above. Note that $-K_Y = \cO_{\bP}(2) \otimes \pi^* \cL$, thus $-K_Y \equiv \cO_{\bP}(2)$ but $-K_Y \not\sim \cO_{\bP}(2)$. Note also that $A$ forms a birationally unbounded family by Claim \ref{claim:unbddK3}. 
\end{eg}

The following gives unbounded examples in the case where $D_m \rightarrow S$ is \'{e}tale in codimension $1$ and $S$ is singular. 

\begin{eg}\label{eg:NikulinAbel}
Let $D$ be a smooth K3 surface with a Nikulin involution $\iota \in \Aut D$, that is, $\iota$ is a symplectic involution 
so that $S:= D / \iota$ is a Du Val K3 surface with 8 $A_1$-singularities $p_1, \ldots, p_8$. 
There are infinitely many components of the moduli space which parametrize K3 surfaces with Nikulin involutions as in \cite[Proposition 2.3]{MR2274533}. 
Let $\pi \colon D \rightarrow S$ be the quotient morphism and $S':= S \setminus \Sing S$ be the smooth part.  
Note that $\pi_* \cO_D \simeq \cO_S \oplus \cL$ for some reflexive sheaf $\cL$ of rank 1 such that 
$\cL^{[2]}:=(\cL^{\otimes 2})^{**} \simeq \cO_S$. 

We can construct a $\bQ$-conic bundle 
$$\bP:= \bP_S( j_* (\Sym (\cO_{S'} \oplus \cL|_{S'})) \rightarrow S,$$
where 
$j \colon S' \rightarrow S$ is an open immersion and $\Sym$ is the symmetric algebra.
We check that $\bP$ has at most $1/2(1,1,1)$-singularities by local computation. 
We also check that $|\cO_{\bP}(2)|$ is a free linear system and contains a smooth irreducible member $\Delta$ 
as in Example \ref{eg:Enriques}.  
We see that $\Delta$ is a K3 surface which can be isomorphic to the original $D$. 
Then the pair $(\bP, \Delta)$ is a plt CY pair such that $K_{\bP}+ \Delta$ is $2$-torsion. We expect that the set of $D$ with Nikulin involutions form a birationally unbounded family.  

We can do the same construction starting from any abelian surface $A$ and its $(-1)$-involution $\iota \in \Aut A$. 
That is, we can construct a $\bQ$-conic bundle $X \rightarrow T:= A/\iota$ with $\Delta \subset X$ so that 
$(X, \Delta)$ is plt, $K_X + \Delta \equiv 0$ and $\Delta \simeq A$ is an abelian surface. 
\end{eg}
	
\begin{rem}
Without the assumption that $D$ is irreducible, the statement is false. 
For example, consider the product $X=S \times \bP^1$ of a K3 surface (or an abelian surface) $S$ and $\bP^1$. 
Note that families of K3 surfaces and abelian surfaces are algebraically unbounded although they are analytically bounded. 

We can also show that the collection of projective K3 surfaces (or abelian surfaces) is birationally unbounded as follows. (This may be well-known, but we include the explanation for the possible convenience of the reader. )


\begin{claim}\label{claim:unbddK3}
Let $\cC:= \{S_d \mid d \in \bZ_{>0} \}$ be the collection of smooth  projective K3 surfaces (or abelian surfaces), where $S_d$ satisfies $\Pic S_d = \bZ \cdot H_d$ and $H_d$ is an ample line bundle of degree $H_d^2 = 2d$.  Then $\cC$ is birationally unbounded. 
\end{claim}

\begin{proof}[Proof of Claim]
The argument is similar as that of \cite[Section 3]{MR1939018}.

Suppose that $\cC$ is birationally bounded. 
Then there exists a projective morphism of algebraic schemes 
$\phi \colon \cS \rightarrow T$ such that, for $d \in \bZ_{>0}$, 
there exist $t_d \in T$ and a birational map $\mu_d \colon \cS_{t_d} \dashrightarrow S_d$ from the fiber $\cS_{t_d}:= \phi^{-1}(t_d)$. Let $T':= \{t_d \mid d\in \bZ_{>0} \} \subset T$ 
and $Z:= \overline{T'} \subset T$ be its closure. Then there exists an irreducible component $Z_i \subset Z$ containing infinitely many $t_d$'s. By considering the base change to $Z_i$, we may assume that $T$ is irreducible and that $T' \subset T$ is dense and contains infinitely many $t_d$'s.  

Let $\eta \in T$ be the generic point and $\cS_{\eta}$ be the generic fiber of $\phi$. 
By taking a resolution of $\cS_{\eta}$ and replacing $T$ by an open subset, we may assume that $\phi \colon \cS \rightarrow T$ is a smooth family of projective surfaces with birational maps $\mu_d \colon \cS_{t_d} \dashrightarrow S_d$ for infinitely many $d$. By running a $K_{\cS}$-MMP over $T$, we may assume that $K_{\cS/T}$ is $\phi$-nef, thus $\mu_d$ is an isomorphism for $d$ with $t_d \in T$. 
Let $\cH$ be a $\phi$-ample line bundle on $\cS$ and $M:= (\cH_t)^2 >0$ be its degree. 
We can take $d \gg 0$ such that $2d > M$ and $t_d \in T$. 
Since $\Pic S_d = \bZ H_d \ni \cH_{t_d}$ and $H_d^2 = 2d > \cH_{t_d}^2 =M >0$, 
this is a contradiction. 
Hence we see that $\cC$ is birationally unbounded. 
\end{proof}


\end{rem}

\begin{rem}
If we only assume that $(X,D)$ is a log canonical pair such that $D$ is irreducible and $K_X+D \sim 0$, then 
such $(X,D)$ forms an unbounded family. For example, we can consider a polarized K3 surface $(S, L)$ of any degree and 
its projective cone $X:= C_p(S,L)$.  
\end{rem}

\begin{rem} 
For any $d >0$, there exists an abelian variety $A$ of dimension $n \ge 2$ with a primitive ample divisor $L$ of type $(1, \ldots, 1, d)$ such that $h^0(A, L) =d$ (and $L^n = n! d$). A general abelian variety of type $(1,1, \ldots, d)$ has the Picard rank $1$. 
Hence abelian varieties of dimension $\ge 2$ are algebraically unbounded (cf. \cite[8.11(1)]{MR2062673}). 

The statement in Theorem \ref{thm:K+D=0case} does not hold when $\dim X \ge 4$. 
Let $X_d:= A_d \times \bP^2$, where $(A_d, L_d)$ is a general abelian variety  
with a primitive polarization $L_d$ of type $(1, \ldots, 1,d)$ as above. 
Then there exists a smooth member $D_d :=A_d \times C \in |{-}K_{X_d}|$ so that $(X_d,D_d)$ is a plt CY pair and $D_d$ is irreducible and reduced,  
where $C \subset \bP^2$ is an elliptic curve. 
Such $D_d$ forms an unbounded family since there is no non-constant map $C \rightarrow A_d$ and $\Pic (A_d \times C) \simeq \Pic A_d \times \Pic C$. 

We can also show such $D_d$ forms a birationally unbounded family by a similar argument as Claim \ref{claim:unbddK3} using the relative MMP guaranteed by \cite[Theorem 1.2]{MR3779687} as follows. 
Suppose that $\{ D_d \mid d \in \bZ_{>0}\}$ is birationally bounded. 
Then, as in Claim \ref{claim:unbddK3}, we can construct a smooth family $\phi \colon \cA \rightarrow T$ over a smooth variety $T$ with 
infinitely many points $t_d \in T$ and a birational map $\mu_d \colon \cA_{t_d} \dashrightarrow D_d$. By \cite[Theorem 1.2]{MR3779687}, we obtain a birational map $\cA \dashrightarrow \cA'$ to a good minimal model $\cA'$ of $\cA$ over $T$ with a morphism $\phi' \colon \cA' \rightarrow T$. Since an abelian variety contains no rational curve and there is no flop on it, we see that $\cA'_{t_d} \simeq D_d$ for $t_d \in T$. Let $\cH'$ be a $\phi'$-ample line bundle on $\cA'$. 
By considering $\cH'|_{\cA'_{t_d}}$ and its pull-back to $A_d$ for sufficiently large $d$, we obtain a contradiction as before.  Hence we obtain the required birational unboundedness.  
\end{rem}

\section{Birational bounded family of Du Val K3 surfaces \\ which are unbounded} 

We consider the following problem in this section. 

\begin{prob}
Let $S$ be a smooth K3 surface with an ample line bundle $L$ with $L^2 =2d$ for a fixed $d >0$. 
Let $S \rightarrow S'$ be a birational morphism onto a normal surface $S'$ (which is a Du Val K3 surface). 
Does there exist an ample line bundle $L'$ on $S'$ with $L'^2 \le N_d$ for some $N_d$ determined by $d$? 
\end{prob}

The following example in the e-mail from Keiji Oguiso is a counterexample to the problem and shows that a birational bounded family of Du Val K3 surfaces can be unbounded. 

\begin{eg}\label{eg:unbdd}
Let $d, m$ be any positive integers. Let $S$ be a polarized K3 surface of degree $2d$ of Picard number $2$ such that ${\rm Pic}\, (S) = \bZ H \oplus \bZ C$ with intersection form 
$$(H^2) = 2d\,\, ,\,\, (H.C) = m\,\, ,\,\, (C^2) = -2\,\, $$
which is constructed in \cite[Theorem 3]{MR1282199} (The $d=2$ case is treated in \cite{MR771073}). 
We know that $H$ is very ample when $d \ge 2$ and $C$ is a $(-2)$-curve by \cite[Lemma 1.2]{MR1282199}.  

We have a contraction $\pi : S \to T$ of $C$ to the rational double point of type $A_1$. Let $L$ be an ample Cartier divisor on $T$. (Note that the local class group of $A_1$ is $\bZ/2$ so that $2L'$ is Cartier for any Weil divisor $L'$ on $T$). Then 
$$\pi^*L = aH + bC,$$ 
where $a$ and $b$ are integers and moreover $a>0$, as $\pi^*L$ is a nef and big Cartier divisor. Since $\pi$ is the contraction of $C$, it follows that $(\pi^*L.C) = 0$. Hence, by $\pi^*L = aH +bC$, we have 
$$a(H.C) + b(C^2) = 0.$$
Substituting $(H.C) = m$ and $(C^2) = -2$ into the equation above, it follows that 
$$b= \frac{am}{2}.$$ 
Also, from $\pi^*L - bC = aH$ with $(\pi^*L.C) = 0$, $(C^2) = -2$ and $(\pi^*L)^2 = (L^2)$ (as $\pi$ is a birational morphism), we have
$$(L^2) - 2b^2 = (\pi^*L -bC)^2 = a^2(H^2) = 2da^2.$$
Hence, for any ample Cartier divisor on $T$, we have
$$(L^2) = 2da^2 + 2b^2 = a^2(2d + \frac{m^2}{2}) \ge 2d + \frac{m^2}{2}.$$
Since $m$ can be taken any positive integer, it follows that the degree of the polarizations on the birational contractions of polarized K3 surfaces of degree $2d$ is unbounded. Hence  contractions of polarized K3 surfaces of a fixed degree do not necessarily form a bounded projective family. 
\end{eg}

\begin{rem}\label{rem:embedquarticcont}
In this remark, we ask whether the surface $S$ and $T$ in Example \ref{eg:unbdd} can be embedded in a rationally connected 3-fold when $d=2$. 


We have an embedding $S \subset \bP^3$ as a quartic surface. 
When $m=2l$ is even, we can construct a 3-fold $\bar{X}$ which contains $T$ as an anticanonical hypersurface so that $(\bar{X}, T)$ is a plt CY pair as follows. 

Assume that $m=2l$ is even. 
By the above consideration, the effective cone $\NEbar(S) \subset \Pic(S) \otimes \bR$ of $S$ is generated by $(-2)$-curves by \cite[Theorem 2]{MR1314742}. 
Hence we can write $\NEbar(S) = \bR_{\ge 0}[C] + \bR_{\ge 0} [ \Gamma]$ for some $(-2)$-curve $\Gamma$. 
Note that $lH-C$ is effective since $(lH-C)^2=-2$ and  $(lH-C)\cdot H = 2l >0$. 
Note that such classes can be reducible in general. 

We show that  $\Gamma \sim lH-C$ as follows. 
Note that we can write $\Gamma= aH-bC$ for some $a,b \in \bZ_{>0}$.  Since we have 
\[
-2=(aH-bC)^2= 4a^2-2mba-2b^2, 
\]
we obtain $a(2a-mb)=2a^2-mba=b^2-1=(b-1)(b+1)$. 
If $b>1$, then we have $2a-mb=2(a-lb)>0$ and 
$$aH-bC = (a-lb) H +b(lH-C) \in \bR_{>0} C + \bR_{>0} (lH-C) \subset \NEbar(S).$$
Hence $aH-bC$ is not on the boundary of $\NEbar(S)$. 
Thus we see that $b=1$ and $a=l$, that is, $\Gamma \sim lH-C$.

Now let $\mu \colon X \rightarrow \bP^3$ be the blow-up along $\Gamma$. 
Let $E_{\Gamma}:= \mu^{-1}(\Gamma)$ be the exceptional divisor and $\tilde{S} \subset X$ be the strict transform of $S$. Let 
\[
L:= \mu^* \cO_{\bP^3}(l^2+1) - l E_{\Gamma}. 
\] 
We see that the restriction $L|_{\tilde{S}} =(l^2+1)H- l(lH-C) =  H+lC$ is the line bundle which induces the birational contraction $\pi \colon S \rightarrow T$ in Example \ref{eg:unbdd}.   

Now assume that $l >4$. 
We see that $L$ is base point free and induces a birational morphism as follows. 
Note that 
\begin{equation}\label{eq:decomp}
L= \mu^*\cO(1) + l(\mu^*\cO(l) - E_{\Gamma}) 
\end{equation} 
and the linear system $|\mu^* \cO(l) - E_{\Gamma}|$ 
contains $\tilde{S} + S_{l-4}$ for all $S_{l-4} \in |\mu^* \cO(l-4)|$.  
Hence the base locus $\Bs L$ of $|L|$ is contained in $\tilde{S}$. 
Since $L|_{\tilde{S}}$ is base point free, we see that $L$ is nef. 
By (\ref{eq:decomp}), we see that $L$ is big. 
Finally, we have an exact sequence 
\[
H^0(X, L) \rightarrow H^0(\tilde{S}, L|_{\tilde{S}}) \rightarrow H^1(X, L-S)=0
\] 
by the Kawamata-Viehweg vanishing since $\tilde{S} \in |{-}K_X|$, 
$L-\tilde{S} = K_X + L$ and $L$ is nef and big. 
This implies that $|L|$ is base point free and induces a birational contraction $\Phi_{L} \colon X \rightarrow \bar{X}$. 
We see that $\Phi_L({\tilde{S}}) \simeq T$. 

We see that $(\bar{X}, T)$ is a plt CY pair although $\bar{X}$ is not $\bQ$-Gorenstein. 
\end{rem}

\begin{rem}\label{rem:misodd}
Let $S \subset \bP^3$ be as in Remark \ref{rem:embedquarticcont} for an odd $m \ge 3$. As in Remark \ref{rem:embedquarticcont}, we see that $\NEbar(S)$ is generated by $C$ and another $-2$-curve $\Gamma$. However, it seems difficult to describe $\Gamma$ explicitly. In order to find such a class, we need to find an integer solution $(a,b)$ of the quadratic equation $$4a^2-2mab -2b^2 =-2$$ with $a,b >0$ .  By a computer program in \cite{Binary}, we find solutions for an explicit $m$. For $m = 15$,  the solutions are $(a,b) = (2G, F-15G)$,  where 
\[
F+ G \sqrt{233} = (2144801346/2 + 140510608/2 \sqrt{233})^n \ \ \text{for } n \ge 0. 
\]
\end{rem}

\section{Some results in higher dimensional case}

We consider the following problem in this section. 

\begin{prob}
Let $n>0$ and $X$ be a normal projective rationally connected $n$-fold with an irreducible $D \in |{-}K_X|$ such that $(X,D)$ is a plt pair (and $D$ is a strict CY variety with only canonical singularities). 
Does such $D$ form a birationally bounded family?  
\end{prob}

\begin{rem}
If $\dim X =4$, then $D$ is a CY 3-fold. By taking a small $\bQ$-factorial modification and running $K_X$-MMP as before, 
we may assume that there is a Mori fiber space $\phi \colon X \rightarrow S$ which induces a surjective morphism $\phi_D:= \phi|_D \colon D \rightarrow S$. 
The problem is to bound this $D$. 

If $\dim S=0$, then $X$ is a $\bQ$-Fano 4-fold with canonical singularities 
and it is bounded (cf. \cite{MR4224714}), thus $D$ is also bounded. 
If $\dim S=2$, then $\phi_D \colon D \rightarrow S$ is an elliptic fibration. Indeed, we check this as in the proof of Proposition \ref{prop:delpezzofibbdd} since its general fiber is an anticanonical member of the general fiber of $\phi$ which is a log del Pezzo surface. Hence $D$ is birationally bounded by Gross' theorem \cite{MR1272978}. 
If $\dim S =3$, then $D \rightarrow S$ is a generically 2:1-cover and 
 $D \rightarrow S$ is branched along a divisor $R \subset S$ since $S$ is rationally connected.  Thus $(S, \frac{1}{2}R)$ is a klt CY pair and $R \in |{-}2K_S|$. 
$S$ is birationally bounded by \cite[Theorem 1.6]{MR4191257} and $D$ is also birationally bounded. 

Hence the problem is reduced to the case $\dim S =1$. However, we don't know how to show the boundedness in these cases. 
\end{rem}

Chen Jiang also pointed out the following. 

\begin{prop}\label{prop:logbddnumequiv}
Let $n,m >0$. 
Let $(X,D)$ be a $n$-dimensional reduced plt CY pair such that 
$X$ is of Fano type.  
Assume that there exists a $\bQ$-divisor $B \neq D$ such that $mB$ is integral, 
$K_X+B \equiv 0$ and $(X, B)$ is lc. 

Then the pair $(X,D)$ is log bounded. 
\end{prop}

\begin{proof}
We see that $(X, D)$ is $\epsilon$-plt for some $\epsilon >0$ by Proposition \ref{prop:differentbounded}(i). 
Then we see that $(X, \frac{1}{2}(B+D))$ is $\epsilon'$-lc for $\epsilon':= \min \{1/2m, \epsilon/2 \}$. 
By \cite[Theorem 1.3]{MR3507257}, we see that $(X,D)$ forms a log bounded family. 
\end{proof}

\begin{rem} The following is pointed out by Yoshinori Gongyo and Roberto Svaldi after the submission to arXiv. 
\begin{prop} Let $(X, D)$ be a reduced plt CY pair such that $X$ is of Fano type. 
\begin{enumerate}
\item[(i)] Then $(X,D)$ is log bitarionally bounded. 
\item[(ii)]  Assume that $X$ is $\bQ$-factorial. Then $(X,D)$ is log bounded. 
\end{enumerate}
\end{prop}

\begin{proof} 
(i) By taking a small $\bQ$-factorial modification, we may assume that $X$ is $\bQ$-factorial. 
Let $\mu \colon X \dashrightarrow X'$ be a birational map induced by a $(-K_X)$-MMP which exists since $X$ is a Mori dream space by \cite{MR2601039}. 
Then we see that $-K_{X'}$ is nef and big. 
Let $D':= \mu_*D$. Note that $\mu$ does not contract $D$ since $D$ is big. 
 Then the pair $(X',D')$ is also a plt CY and $\epsilon$-plt for some $\epsilon >0$ by Proposition \ref{prop:differentbounded}. 
By these, we see that $X'$ is an $\epsilon$-lc weak Fano variety, thus it is bounded by \cite{MR4224714}.  
Hence $D' \equiv -K_{X'}$ is also bounded.  

(ii) We also use the notation in (i). Then $(X',D')$ is $\epsilon$-plt and $-K_{X'}$ is nef and big. 
Thus we can take a positive integer $m$ determined by $\dim X$ such that $-mK_{X'}$ is base point free. 
Then, by taking a general member of $A' \in |{-}mK_{X'}|$ and putting $B':= \frac{1}{m} A'$, we obtain a $\frac{1}{m}$-lc CY pair $(X',B')$. Moreover, $K_{X'} +B'$ is an $m$-complement (cf. \cite[2.18]{MR3997127}). 
Then we obtain an $m$-complement $K_X+B$ as in \cite[6.1(3)]{MR3997127}, where $B$ is the sum of the strict transform of $B'$ and some effective divisor supported on the $\mu$-exceptional divisors.  
Hence, by Proposition \ref{prop:logbddnumequiv}, we see that $(X, D)$ is log bounded.   
\end{proof}
\end{rem}

Johnson--Koll\'{a}r \cite{Johnson-Kollar} proved that there are only finitely many quasismooth weighted CY hypersurfaces of fixed dimension. 
Chen \cite{MR3356738} proved that there are only finitely many families of CY weighted complete intersections. 
CY varieties in toric varieties are often considered in mirror symmetry and so on. 
Although toric varieties are unbounded, we can show the following. 

\begin{cor}\label{cor:torichypersurface}
Let $X$ be a normal projective  toric variety with $D \in |{-}K_X|$ with only canonical singularities. 
Then $(X, D)$ form a log bounded family.  (Thus both $X$ and $D$ are bounded. )
\end{cor}

\begin{proof}
Note that, since $X$ is toric and $\bQ$-Gorenstein in codimension 2, 
we see that $X$ has only canonical singularities in codimension 2 by \cite[Theorem 5]{MR1027448}. 
Thus $D$ is Cartier in codimension 2 and $(X,D)$ is plt by inversion of adjunction (cf. \cite[Theorem 5.50]{MR1658959}). 

Let $\Delta \subset X$ be the union of toric divisors. 
Then we see that $(X,\Delta)$ is lc. 
By applying Proposition \ref{prop:logbddnumequiv}, we obtain the claim. 
\end{proof}

\section*{Acknowledgement}
The author would like to thank Keiji Oguiso for valuable discussions through e-mails and allowing him to use the argument in Example \ref{eg:unbdd}. 
He also thanks Chen Jiang for pointing out mistakes and improvements in the manuscript. He is grateful to Yoshinori Gongyo, Kenji Hashimoto and Roberto Svaldi for valuable comments. 
The author is grateful to the anonymous referees for careful reading and many valuable comments. 	
	This work was partially supported by JSPS KAKENHI Grant Numbers JP17H06127, JP19K14509. 
	
\bibliographystyle{amsalpha}
\bibliography{sanobibs-bddlogcy}

\end{document}